\documentclass[12pt]{article}
\usepackage[fleqn]{amsmath}
\usepackage[cp1251]{inputenc}
\usepackage[russian,english]{babel}
\usepackage{amsmath,amsfonts,amssymb,graphicx,makeidx,amsthm}
\usepackage{amscd,amsfonts,amssymb,multicol}
\usepackage[T2A]{fontenc}
\usepackage[cp1251]{inputenc}
\usepackage[english,russian]{babel}
\usepackage{euscript}
\newtheorem{theorem}{Theorem}
\newtheorem{lemma}{Lemma}

\newtheorem{rem}{Remark}
\newtheorem{Def}{Defenition}

\begin{document}

\begin{center}
\textbf{ON WEIGHTED GENERALIZED FUNCTIONS ASSOCIATED WITH  QUADRATIC FORMS}
\end{center}

\begin{center}
E. L. Shishkina
\end{center}
MSC 2010 46T12, 46F05\\ 

\emph{Keywords:} weighted generalized function, quadratic form,  ultra-hyperbolic operator, Bessel operator.

\emph{Abstract.} In this article we  consider certain types of weighted generalized functions
associated with nondegenerate quadratic forms.  Such functions and their derivatives are used for constructing fundamental solutions of iterated ultra-hyperbolic equations with Bessel operator and for constructing   negative real powers of  ultra-hyperbolic operators  with Bessel operator.
\vskip 1cm

{\bf 1. Introduction and main definitions. }

The weighted generalized functions associated with nondegenerate indefinite quadratic form considered in this article are necessary for construction of the ultra-hyperbolic Riezs potential with Bessel operator. Riezs potential with Bessel operator and other operators with the Bessel differential operator  are very interesting subjects with many applications (see, for example,  \cite{KiprIv}-\cite{Tur}).

We deal with the part of Euclidean space
$$
\mathbb{R}^+_n{=}\{x{=}(x_1,\ldots,x_n)\in\mathbb{R}_n,\,\,\, x_1{>}0,\ldots, x_n{>}0\}.
$$
Let $\Omega$ be finite or infinite open set in $\mathbb{R}_n$ symmetric with respect  to each hyperplane $x_i{=}0$, $i=1,...,n$, $\Omega_+=\Omega\cap{\mathbb{R}}\,^+_n$ and $\overline{\Omega}_+=\Omega\cap\overline{\mathbb{R}}\,^+_n$ where $$
\overline{\mathbb{R}}\,^+_n{=}\{x{=}(x_1,\ldots,x_n)\in\mathbb{R}_n,\,\,\, x_1{\geq}0,\ldots, x_n{\geq}0\}.
$$ We have $\Omega_+\subseteq{\mathbb{R}}\,^+_n$ and $\overline{\Omega}_+\subseteq\overline{\mathbb{R}}\,^+_n$.

We consider the class $C^\infty(\Omega_+)$ consisting of infinitely differentiable on $\Omega_+$ functions.
We denote the subset of functions from $C^\infty(\Omega_+)$ such that all derivatives  of these functions with respect to $x_i$ for any $i=1,...,n$
 are continuous up to $x_i{=}0$ by $C^\infty(\overline{\Omega}_+)$. Function $f\in C^\infty(\overline{\Omega}_+)$ we will call \emph{even with respect to} $x_i$, $i=1,...,n$ if $\frac{\partial^{2k+1}f}{\partial x_i^{2k+1}}\biggr|_{x=0}=0$ for all nonnegative integer $k$ (see \cite{Kipriyanov}, p. 21). Class $C^\infty_{ev}(\overline{\Omega}_+)$ consists of functions from $C^\infty(\overline{\Omega}_+)$ even with respect to each variable $x_i$, $i=1,...,n$.
Let ${\stackrel{\circ}C}\,^\infty_{ev}(\overline{\Omega}_+)$ be the space of all functions  $f\in C^\infty(\overline{\Omega}_+)$ with a compact support. Elements of ${\stackrel{\circ}C}\,^\infty_{ev}(\overline{\Omega}_+)$ we will call \emph{test functions} and use the notation ${\stackrel{\circ}C}\,^\infty_{ev}(\overline{\Omega}_+)=\mathcal{D}_+(\overline{\Omega}_+)$.

We define $K$ as an arbitrary compact in $\mathbb{R}_n$ symmetric with respect  to each hyperplane $x_i{=}0$, $i=1,...,n$, $K_+=K\cap\overline{\mathbb{R}}\,^+_n$.
The \emph{distribution} $u$ on $\overline{\Omega}_+$ is the linear form on $\mathcal{D}_+(\overline{\Omega}_+)$ such that for all compacts $K_+\subset \overline{\Omega}_+$,  constants $C$ and $k$ exist and
$$
|u(f)|\leq C\sum\limits_{|\alpha|\leq k}\sup|{D}^\alpha f|,\qquad f\in {\stackrel{\circ}C}\,^\infty_{ev}(K_+),
$$
where ${D}^\alpha={D}^{\alpha_1}_{x_1}...{D}^{\alpha_n}_{x_n}$, $\alpha=(\alpha_1,...,\alpha_n)$, $\alpha_1,...,\alpha_n$ are integer nonnegative numbers, ${D}_{x_j}=i\frac{\partial}{\partial x_j}$,  $i$ is imaginary unit, $j=1,...,n$. The set of all  distributions on the set $\overline{\Omega}_+$ is denoted by $\mathcal{D}_+'(\overline{\Omega}_+)$ (see \cite{Kipriyanov}, p. 11 and \cite{Hormander}, p. 34).

Multiindex $\gamma{=}(\gamma_1,{\ldots},\gamma_{n})$ consists of positive fixed reals $\gamma_i>0$, $i{=}1,{...},n$ and   $|\gamma|{=}\gamma_1{+}{\ldots}{+}\gamma_{n}.$
Let $L_p^{\gamma}(\Omega_+)$, $1\leq p<\infty$ be the space of  all measurable in $\Omega_+$ functions even with respect to each variable $x_i$, $i=1,...,n$ such that
$$
\int\limits_{\Omega_+}|f(x)|^p x^\gamma dx<\infty,\qquad \quad
 x^\gamma=\prod\limits_{i=1}^n x_i^{\gamma_i}.
$$
For a real number $p\geq 1$, the $L_p^\gamma(\Omega_+)$--norm of $f$ is defined by
$$
||f||_{L_p^\gamma(\Omega_+)}=\left(\,\,\int\limits_{\Omega_+}|f(x)|^p x^\gamma dx\right)^{1/p}.
$$

Weighted measure of $\Omega_+$ is denoted by
  ${\rm{mes}}_\gamma(\Omega)$ and is defined by formula
$$
{\rm{mes}}_\gamma(\Omega_+)=\int\limits_{\Omega_+}x^\gamma dx.
$$
For every measurable function $f(x)$ defined on $\mathbb{R}_n^+$ we consider
$$
\mu_\gamma(f,t)={\rm{mes}}_\gamma\{x\in\mathbb{R}_n^+:\,|f(x)|>t\}=\int\limits_{\{x:\,\,|f(x)|>t\}^+}x^\gamma dx
$$
where $\{x:\,\,|f(x)|>t\}^+{=}\{x\in\mathbb{R}_n^+:\,|f(x)|>t\}$. We will call the function   $\mu_\gamma=\mu_\gamma(f,t)$ a {\it{weighted distribution function}} $|f(x)|$.

A space $L_\infty^\gamma(\Omega_+)$ is defined as a set of measurable on  $\Omega_+$ and even with respect to each variable function $f(x)$ such as
$$
||f||_{L_\infty^\gamma(\Omega_+)}=\underset{x\in \Omega_+}{{\rm ess\,sup}_\gamma}|f(x)|
=\inf\limits_{a\in\Omega_+}\{\mu_\gamma(f,a)=0\}<\infty.
$$
For $1\leq p\leq\infty$  the $L_{p,loc}^\gamma(\Omega_+)$ is the set of functions $u(x)$ defined almost everywhere in $\Omega_+$ such that $uf\in L_{p}^\gamma(\Omega_+)$ for any $f\in{\stackrel{\circ}C}\,^\infty_{ev}(\overline{\Omega}_+)$.
Each function $u(x)\in L_{1,loc}^\gamma(\Omega_+)$ will be identified with the functional $u\in \mathcal{D}_+'(\overline{\Omega}_+)$
acting according to the formula
\begin{equation}\label{RegDist}
(u,f)_\gamma=\int\limits_{\mathbb{R}^+_n} u(x)\,f(x)\,x^\gamma\, dx,\quad x^\gamma=\prod_{i=1}^n x_i^{\gamma_i},\quad f\in {\stackrel{\circ}C}\,^\infty_{ev}(\overline{\mathbb{R}}\,^+_n).
\end{equation}
Functionals $u\in \mathcal{D}_+'(\overline{\Omega}_+)$ acting by the formula \eqref{RegDist} will be called  \emph{regular weighted functionals}.
All other functionals $u\in \mathcal{D}_+'(\overline{\Omega}_+)$ will be called  \emph{singular weighted functionals}.

{\bf 2.Weighted generalized functions concentrated on the part of the cone.}
In this section we consider weighted generalized functions $\delta_\gamma(P)$ concentrated on the part of the cone and give formulas for its derivatives.

Generalized function $\delta_\gamma$ is defined by the equality  (by analogy with \cite{Gelfand} p. 247)
$$
(\delta_\gamma,\varphi)_\gamma=\varphi(0),\quad \varphi(x)\in K^+.
$$
For convenience we will write
$$
(\delta_\gamma,\varphi)_\gamma=\int\limits_{\mathbb{R}_n^+}\delta_\gamma(x)\varphi(x)x^\gamma dx=\varphi(0).
$$

Let $p,q{\in}\mathbb{N}$,  $n{=}p{+}q$
 and
$$
P=|x'|^2-|x''|^2=x_1^2+...+x_p^2-x_{p+1}^2-...-x_{p+q}^2,
$$
where $x{=}(x_1,{...},x_n){=}(x',x''){\in}\mathbb{R}_n^+,$ $x'{=}(x_1,{...},x_p)$, $x''{=}(x_{p+1},{...},x_{p+q})$.

\begin{Def}  Let $\varphi{\in}\mathcal{D}_+(\overline{\mathbb{R}}\,_n^+)$ vanishes at the origin.  For such $\varphi$ we define generalized function $\delta_\gamma(P)$ concentrated on the part of the cone $P{=}0$ belonging to  $\mathbb{R}_n^+$ by the formula
\begin{equation}\label{Seq1Eq01}
(\delta_\gamma(P),\varphi)_\gamma=\int\limits_{\mathbb{R}_n^+}\delta_\gamma(|x'|^2-|x''|^2)\varphi(x)x^\gamma
dx.
\end{equation}
\end{Def}

If the function $\varphi{\in}\mathcal{D}_+(\overline{\mathbb{R}}\,_n^+)$ doesn't vanish at the origin then $(\delta_\gamma(P),\varphi)_\gamma$ is defined by regularizing the integral.

\begin{lemma} Let $\varphi{\in}\mathcal{D}_+(\overline{\mathbb{R}}\,_n^+)$ vanishes at the origin, $p{>}1$ and $q{>}1$.
For $\delta_\gamma(P)$ the representation
\begin{equation}\label{Seq1Eq02}
(\delta_\gamma(P),\varphi)_\gamma=\frac{1}{2}\int\limits_{0}^{\infty}\int\limits_{S_p^+}\int\limits_{S_q^+}\varphi(s\,\omega)s^{n+|\gamma|-3}\omega^\gamma
dS_pdS_q ds
\end{equation}
holds true. In \eqref{Seq1Eq02}  $\omega{=}(\omega',\omega'')$, $\omega'{=}(\omega_1,{...},\omega_p){\in}\mathbb{R}^+_p$, $\omega''{=}(\omega_{p{+}1},{...},\omega_{p{+}q}){\in}\mathbb{R}^+_q$, $n{=}p{+}q$, $|\omega'|{=}|\omega''|{=}1$, $\omega^\gamma{=}\prod\limits_{i=1}^n\omega_i^{\gamma_i}$, $dS_p$ and $dS_q$ are  the elements of surface area on the part of unit sphere  $$S_p^+=\{\omega'\in\mathbb{R}^+_p:|\omega'|{=}1\}\quad \text{и}\quad S_q^+{=}\{\omega''\in\mathbb{R}^+_q:|\omega''|=1\},$$ respectively.
For the k-th derivative  ($k{\in}\mathbb{N}$) of $\delta_\gamma(P)$ we have
\begin{equation}\label{Seq1Eq04}
(\delta^{(k)}_\gamma(P),\varphi)_\gamma=\int\limits_{0}^{\infty}\left[\left(\frac{1}{2s}\frac{\partial}{\partial
s}\right)^k\psi(r,s)s^{q+|\gamma''|-2}\right]_{s=r}r^{{p+|\gamma'|}-1}dr.
\end{equation}
where
\begin{equation}\label{Seq1Eq05}
\psi(r,s)=\frac{1}{2}\int\limits_{S_p^+}\int\limits_{S_q^+}
\varphi(r\omega',s\omega'') \omega^\gamma dS_pdS_q.
\end{equation}
\end{lemma}

\begin{proof} Let us transform \eqref{Seq1Eq01} to bipolar coordinates defined by
\begin{equation}\label{Bip}
x_1=r\omega_1,...,x_p=r\omega_p,\,x_{p+1}=s\omega_{p+1},...,x_{p+q}=s\omega_{p+q},
\end{equation}
where
$$
r=\sqrt{x_1^2+...+x_p^2},\quad s=\sqrt{x_{p+1}^2+...+x_{p+q}^2},
$$
$$|\omega'|=\sqrt{\omega_1^2+...+\omega_{p}^2}=1,\quad |\omega''|=\sqrt{\omega_{p+1}^2+...+\omega_{p+q}^2}=1.$$
We obtain
$$
(\delta_\gamma(P),\varphi(x))_\gamma=
$$
$$
=\int\limits_{0}^{\infty}\int\limits_{0}^{\infty}\int\limits_{S_p^+}\int\limits_{S_q^+}\delta_\gamma(r^2-s^2)\varphi(r\omega',s\omega'')r^{p+|\gamma'|-1}s^{q+|\gamma''|-1}\omega^\gamma
dS^p_1dS^q_1 drds.
$$

Now let us choose the coordinates to be
 $r^2=u$, $s^2=v$. In these coordinates we have
$$
(\delta_\gamma(P),\varphi)_\gamma=\frac{1}{4}\int\limits_{0}^{\infty}\int\limits_{0}^{\infty}\int\limits_{S_p^+}\int\limits_{S_q^+}\delta_\gamma(u-v)\varphi(\sqrt{u}\omega',\sqrt{v}\omega'')u^{\frac{p+|\gamma'|}{2}-1}\times
$$
$$\times v^{\frac{q+|\gamma''|}{2}-1}\omega^\gamma
dS_pdS_q dudv
=\frac{1}{4}\int\limits_{0}^{\infty}\int\limits_{S_p^+}\int\limits_{S_q^+}\varphi(\sqrt{v}\omega)v^{\frac{n+|\gamma|}{2}-2}\omega^\gamma
dS_pdS_q dv.
$$
Returning to the variable $s$ by the formula $v{=}s^2$,
we obtain \eqref{Seq1Eq02}.

Now we prove the formula \eqref{Seq1Eq04}. After the change of variables by
\eqref{Bip} and $r^2{=}u$, $s^2{=}v$ in   $(\delta_\gamma^{(k)}(P),\varphi)_\gamma$  we get
$$
(\delta^{(k)}_\gamma(P),\varphi)_\gamma=
\frac{1}{4}\int\limits_{0}^{\infty}\int\limits_{0}^{\infty}\int\limits_{S_p^+}\int\limits_{S_q^+}\frac{\partial^k}{\partial
v^k}[\delta_\gamma(v-u)]\varphi(\sqrt{u}\omega',\sqrt{v}\omega'')\times
$$
$$\times u^{\frac{p+|\gamma'|}{2}-1}v^{\frac{q+|\gamma''|}{2}-1}\omega^\gamma
dS_pdS_q dudv=\int\limits_{0}^{\infty}\int\limits_{0}^{\infty}\int\limits_{S_p^+}\int\limits_{S_q^+}\delta_\gamma(v-u)\times
$$
$$\times\frac{(-1)^k}{4}\frac{\partial^k}{\partial
v^k}\left[\varphi(\sqrt{u}\omega',\sqrt{v}\omega'')v^{\frac{q+|\gamma''|}{2}-1}\right]u^{\frac{p+|\gamma'|}{2}-1}\omega^\gamma dS_pdS_q dudv=$$
$$
\frac{(-1)^k}{4}\int\limits_{0}^{\infty}\int\limits_{S_p^+}\int\limits_{S_q^+}u^{\frac{p+|\gamma'|}{2}-1}\omega^\gamma\left[\frac{\partial^k}{\partial
v^k}\varphi(\sqrt{u}\omega',\sqrt{v}\omega'')v^{\frac{q+|\gamma''|}{2}-1}\right]_{v=u}
dS_pdS_q du.
$$
Returning to the variables $r$, $s$ and using notation  \eqref{Seq1Eq05} we obtain \eqref{Seq1Eq04}. This completes the proof of Lemma 1.
\end{proof}

\begin{rem}
Similarly, we can get the formula
\begin{equation}\label{Seq1Eq07}
(\delta_\gamma^{(k)}(P),\varphi)_\gamma{=}(-1)^k\int\limits_{0}^{\infty}\left[\left(\frac{1}{2r}\frac{\partial}{\partial
r}\right)^k\psi(r,s)r^{p+|\gamma'|{-}2}\right]_{r{=}s}s^{{q+|\gamma''|}{-}1}ds.
\end{equation}
\end{rem}

\begin{rem}
Noticing that when $k{=}0$ formulas \eqref{Seq1Eq04} and  \eqref{Seq1Eq07} are equivalent to the formula \eqref{Seq1Eq02} we will examine intergals  \eqref{Seq1Eq04} and  \eqref{Seq1Eq07} at $k{\in}\mathbb{N}\cup\{0\}$.
\end{rem}

Let $\varphi{\in}\mathcal{D}_+(\overline{\mathbb{R}}\,_n^+)$. Assuming that the function $\varphi$ vanishes at the origin we have that integrals  \eqref{Seq1Eq04} and  \eqref{Seq1Eq07} converge for all $k\in\mathbb{N}\cup\{0\}$. If  the function $\varphi$ doesn't vanish at the origin  then
 integrals \eqref{Seq1Eq04} and   \eqref{Seq1Eq07} converge only for  $k<\frac{p+q+|\gamma|-2}{2}$. In this case for $k\geq\frac{p+q+|\gamma|-2}{2}$ we will consider the regularization of \eqref{Seq1Eq04} and  \eqref{Seq1Eq07}
denoting them $\delta^{(k)}_{\gamma,1}(P)$
and $\delta^{(k)}_{\gamma,2}(P)$, respectively. So using the expression \eqref{Seq1Eq05} for $p>1$, $q>1$ and $k{\in}\mathbb{N}\cup\{0\}$ we have
\begin{equation}\label{Seq1Eq08}
   ( \delta^{(k)}_{\gamma,1}(P),\varphi)_\gamma=\int\limits_{0}^{+\infty}\left[\left(\frac{1}{2s}\frac{\partial}{\partial
s}\right)^k\psi(r,s)s^{q+|\gamma''|-2}\right]\biggr|_{s=r}r^{{p+|\gamma'|}-1}dr,
\end{equation}
\begin{equation}\label{Seq1Eq09}
   ( \delta^{(k)}_{\gamma,2}(P),\varphi)_\gamma=(-1)^k\int\limits_{0}^{+\infty}\left[\left(\frac{1}{2r}\frac{\partial}{\partial
r}\right)^k\psi(r,s)r^{p+|\gamma'|-2}\right]\biggr|_{r=s}s^{{q+|\gamma''|}-1}ds.
\end{equation}
The integrals \eqref{Seq1Eq08} and \eqref{Seq1Eq09} converge and coincide for $k<\frac{p+q+|\gamma|-2}{2}$ and for $k\geq\frac{p+q+|\gamma|-2}{2}$ these integrals must be understood in the sense of their regularizations.

{\bf 2.Weighted generalized function $P^\lambda_{\gamma,+}$.}
Let $n{=}p{+}q,$ $p{>}1,$ $q{>}1$ and $P(x)=x_1^2+...+x_p^2-x_{p+1}^2-...-x_{p+q}^2$.
Here and further let $\varphi\in\mathcal{D}_+(\overline{\mathbb{R}}\,_n^+)$. We define the weighted generalized function
 $P_{\gamma,+}^\lambda$ by
\begin{equation}\label{Seq2Eq12}
(P_{\gamma,+}^\lambda,\varphi)_\gamma=\int\limits_{\{P(x)>0\}^+}P^\lambda(x)\varphi(x)x^\gamma
dx,
\end{equation}
where $\{P(x)>0\}^+=\{x\in\mathbb{R}_n^+:P(x)>0\}$, $\lambda\in\mathbb{C}$.

Weighted generalized function $P^\lambda_{\gamma,+}$  and its derivatives are used for constructing fundamental solutions of iterated B-ultra-hyperbolic equations of the form
 $
 L_B^ku=f(x),\quad k\in\mathbb{N},\quad x\in\mathbb{R}_n,\quad x_i>0,\quad i=1,...,n,
 $
where $L_B$ is B-ultra-hyperbolic operator (see \cite{Tur} and \cite{Gruter}--\cite{DU})
 $$
 L_B=B_{x_1}{+}{...}{+}B_{x_p}{-}B_{x_{p+1}}{-}{...}{-}B_{x_n},
 $$
 $B_{x_i}{=}\frac{\partial^2}{\partial x_i^2}{+}\frac{\gamma_i}{x_i}\frac{\partial}{\partial x_i}$ is the Bessel operator, $\gamma_i{>}0$, $i{=}1,{...},n$.

It should also be noted that negative real powers of an operator  $L_B$ called generalized B-hyperbolic potentials (see \cite{ShishkinaAMADE})  are
constructed using  function $P^\lambda_{\gamma,+}$. Let us find  singularities of $(P_{\gamma,+}^\lambda,\varphi)_\gamma$. For this purpose we transform \eqref{Seq2Eq12}
to bipolar coordinates \eqref{Bip} and  using notation  \eqref{Seq1Eq05} for integral \eqref{Seq2Eq12}  we obtain
\begin{equation}\label{Seq2Eq14}
(P_{\gamma,+}^\lambda,\varphi)_\gamma=\int\limits_{0}^\infty\int\limits_{0}^r(r^2-s^2)^\lambda
\psi(r,s)r^{p+|\gamma'|-1}s^{q+|\gamma''|-1} drds.
\end{equation}
We now make in  \eqref{Seq2Eq14} change of variables $u{=}r^2,$ $v{=}s^2$:
$$
(P_{\gamma,+}^\lambda,\varphi)_\gamma=\frac{1}{4}\int\limits_{0}^\infty\int\limits_{0}^u(u-v)^\lambda
\psi_1(u,v)u^{\frac{p+|\gamma'|}{2}-1}s^{\frac{q+|\gamma''|}{2}-1}
dudv,
$$
where $\psi_1(u,v){=}\psi(r,s)$ when $u{=}r^2,$ $v{=}s^2$.

If we write  $v{=}ut$ then we  obtain
\begin{equation}\label{Seq2Eq16}
(P_{\gamma,+}^\lambda,\varphi)_\gamma=\int\limits_{0}^\infty
u^{\lambda+\frac{p+q+|\gamma|}{2}-1}\Phi(\lambda,u)du,
\end{equation}
where
\begin{equation}\label{Phi}
\Phi(\lambda,u)=\frac{1}{4}\int\limits_{0}^1
 (1-t)^\lambda t^{\frac{q+|\gamma''|}{2}-1}
\psi_1(u,tu)dt.
\end{equation}

The formula \eqref{Seq2Eq16}  shows that $P_{\gamma,+}^\lambda$  has two sets of poles. The first of these consists of  poles of
 $\Phi(\lambda,u)$.
Namely for $t{=}1$ function $\Phi(\lambda,u)$ has singularity when
\begin{equation}\label{Seriya1}
\lambda=-1,-2,...,-k,...
\end{equation}
in which $\Phi(\lambda,u)$ has simple poles with residues
\begin{equation}\label{Seq2Eq17}
\mathop{\rm res}_{\lambda=-k}\Phi(\lambda,u)=\frac{1}{4}\frac{(-1)^{k-1}}{(k-1)!}\frac{\partial^{k-1}}{\partial
t^{k-1}}\left[t^{\frac{q+|\gamma''|-2}{2}}\psi_1(u,tu)\right]_{t=1}.
\end{equation}

Moreover integral  \eqref{Seq2Eq16} has  poles at the points
\begin{equation}\label{Seriya2}
\lambda=-\frac{n+|\gamma|}{2},\,-\frac{n+|\gamma|}{2}-1,\,...,\,-\frac{n+|\gamma|}{2}-k,\,...,
\end{equation}
where $n=p+q$, $\gamma=(\gamma',\gamma'')$. Wherein
\begin{equation}\label{Res2}
\mathop{\rm res}_{\lambda=-\frac{n+|\gamma|}{2}-k}(P_{\gamma,+}^\lambda,\varphi)_\gamma=\frac{1}{k!}\frac{\partial^{k}}{\partial
u^{k}}\Phi\left(-\frac{n+|\gamma|}{2}-k,u\right)\biggr|_{u=0}.
\end{equation}

We have three cases. The first case is when singular point  $\lambda$ belongs to the first set \eqref{Seriya1}, but
not to the second \eqref{Seriya2}. The second case is when singular point  $\lambda$  belongs to the second \eqref{Seriya2}, but $\lambda{\neq}{-}k$, $k{\in}\mathbb{N}$. And the third case is when $\lambda$ belongs to the first set \eqref{Seriya1} and belongs to the second set  \eqref{Seriya2}.
Let us now study each case separately in the following three theorems.

\begin{theorem}
If $\lambda{=}{-}k$, $k\in\mathbb{N}$ and $n+|\gamma|{\in}\mathbb{R}\backslash \mathbb{N}$  or $n+|\gamma|{\in}\mathbb{N}$ and $n+|\gamma|{=}2k-1$, $k{\in}\mathbb{N}$
and also if $n{+}|\gamma|$ is even and $k{<}\frac{n+|\gamma|}{2}$ the weighted generalized function $P_{\gamma,+}^\lambda$ has simple pole with residue
\begin{equation}\label{Res1}
\underset{\lambda=-k}{\rm res}P_{\gamma,+}^\lambda=\frac{(-1)^{k-1}}{(k-1)!}\delta_{\gamma,1}^{(k-1)}(P).
\end{equation}
\end{theorem}

\begin{proof}
Let us write  $\Phi(\lambda,u)$ in the neighborhood of $\lambda=-k$  in the form
$$
\Phi(\lambda,u)=\frac{\Phi_0(u)}{\lambda+k}+\Phi_1(\lambda,u),
\quad \Phi_0(u)=\mathop{\rm res}_{\lambda=-k}\Phi(\lambda,u),$$
where function $\Phi_1(\lambda,u)$  is regular at $\lambda{=}{-}k$. We obtain
\begin{equation}\label{Seq2Eq22}
(P_{\gamma,+}^\lambda,\varphi)_\gamma{=}\frac{1}{\lambda{+}k}\int\limits_{0}^\infty
u^{\lambda{+}\frac{n{+}|\gamma|}{2}{-}1}\Phi_0(u)du{+}\int\limits_{0}^\infty
u^{\lambda{+}\frac{n{+}|\gamma|}{2}{-}1}\Phi_1(\lambda,u)du.
\end{equation}
The integrals in \eqref{Seq2Eq22} are regular functions of $\lambda$ at $\lambda{=}{-}k$. Therefore $(P_{\gamma,+}^\lambda,\varphi)_\gamma$ has a simple
pole at such a point and using \eqref{Seq2Eq17} we have
\begin{equation}\label{Seq2Eq23}
\underset{\lambda{=}{-}k}{\rm res}(P_{\gamma,+}^\lambda,\varphi){=}\frac{(-1)^{k{-}1}}{4(k{-}1)!}\int\limits_{0}^\infty
u^{\frac{n+|\gamma|}{2}-k-1}\frac{\partial^{k-1}}{\partial
t^{k-1}}\left[t^{\frac{q+|\gamma''|}{2}{-}1}\psi_1(u,tu)\right]_{t=1}du.
\end{equation}
If in \eqref{Seq2Eq23} we get $tu=v$ then we may write
\begin{equation}\label{For1}
\underset{\lambda=-k}{\rm res}(P_{\gamma,+}^\lambda,\varphi)=\frac{(-1)^{k-1}}{4(k{-}1)!}\int\limits_{0}^\infty
\frac{\partial^{k-1}}{\partial v^{k-1}}\left[v^{\frac{q+|\gamma''|}{2}-1}\psi_1(u,v)\right]_{v=u}u^{\frac{p+|\gamma'|}{2}-1}du,
\end{equation}
where the integral is to be understood in the sense of its regularization
for $k\geq\frac{n}{2}$.
We now make the change of variables $u=r^2$ и $v=s^2$ in  \eqref{Seq1Eq08} and will have
\begin{equation}\label{Seq2Eq24}
( \delta^{(k-1)}_{\gamma,1}(P),\varphi)_\gamma=\frac{1}{2}\int\limits_{0}^{\infty}\left[\frac{\partial^{k-1}}{\partial
v^{k-1}}v^{\frac{q+|\gamma''|}{2}-1}\psi_1(u,v)\right]_{v=u}u^{\frac{{p+|\gamma'|}}{2}-1}du,
\end{equation}
where
$$
\psi_1(u,v)=\frac{1}{2}\int\limits_{S_p^+}\int\limits_{S_q^+}
\varphi(\sqrt{u}\omega',\sqrt{v}\omega'') \omega^\gamma dS_pdS_q.
$$
Formulas  \eqref{For1} and \eqref{Seq2Eq24} imply \eqref{Res2}. For $k\geq\frac{n}{2}$ integral in  \eqref{Seq2Eq24} is to be understood in the sense of its regularization. In the case when  $n+|\gamma|{\in}\mathbb{R}\backslash \mathbb{N}$  or $n+|\gamma|{\in}\mathbb{N}$ and $n+|\gamma|{=}2k-1$, $k{\in}\mathbb{N}$ regularization of the integral in  \eqref{Seq2Eq24}
 is defined by analytic continuation. This proves the desired result.
\end{proof}

\vskip 0.5cm

Now we will study the case when  the singular point $\lambda$ is in the second set \eqref{Seriya2}, but not in the first \eqref{Seriya1}.
If $\lambda{=}{-}\frac{n+|\gamma|}{2}{-}k$, $k{=}0,1,2,...$ and $n+|\gamma|{\in}\mathbb{R}\backslash \mathbb{N}$  or $n+|\gamma|{\in}\mathbb{N}$ and $n+|\gamma|{=}2k-1$, $k{\in}\mathbb{N}$  then function $\Phi(\lambda,u)$ is regular
in the neighborhood of $\lambda{=}{-}\frac{n+|\gamma|}{2}{-}k$. Therefore function
$
(P_{\gamma,+}^\lambda,\varphi)_\gamma
$
will have a simple pole with residue given by   \eqref{Res2}.

Before proceeding to the expression of the residue  $\underset{\lambda=-\frac{n+|\gamma|}{2}-k}{\rm res}(P_{\gamma,+}^\lambda,\varphi)$ through derivatives of function $\varphi(x)$ at the origin we will obtain one useful formula.
Consider the B-ultra-hyperbolic differential operator
$$
L_B=B_{\gamma_1'}+...+B_{\gamma_p'}-B_{\gamma_{p+1}''}-B_{\gamma_{p+q}''},\qquad B_{\gamma_i}=\frac{\partial^2}{\partial x_i^2}+\frac{\gamma_i}{x_i}\frac{\partial}{\partial x_i}.
$$
Applying an operator $L_B$ to quadratic form
$$
P(x){=}x_1^2{+}{...}{+}x_p^2{-}x_{p+1}^2{-}{...}{-}x_{p+q}^2,\,\,n=p+q,\,\,p>1,\,\,q>1
$$ we obtain
\begin{equation}\label{Seq2Eq28}
    L_BP^{\lambda+1}(x){=}4(\lambda{+}1)\left(\lambda{+}\frac{n{+}|\gamma|}{2}\right)P^{\lambda}(x).
\end{equation}

\begin{theorem}
Let $n+|\gamma|$ be not integer or $n+|\gamma|{\in}\mathbb{N}$ and $n+|\gamma|{=}2k-1$, $k{\in}\mathbb{N}$. When а $p+|\gamma'|$ is not integer or $p{+}|\gamma'|{\in}\mathbb{N}$, $p{+}|\gamma'|{=}2m{-}1,$ $m{\in}\mathbb{N}$ and $q+|\gamma''|$ is even weighted functional $P_{\gamma,+}^\lambda$	 has simple poles at $\lambda{=}{-}\frac{n{+}|\gamma|}{2}{-}k$,
$k{\in}\mathbb{N}{\cup}\{0\}$ with residues
$$
\underset{\lambda=-\frac{n+|\gamma|}{2}-k}{\rm res}P_{\gamma,+}^{\lambda}=\frac{(-1)^{\frac{q{+}|\gamma''|}{2}}}{2^{n+2k}k!}\frac{\prod\limits_{i{=}1}^n\Gamma\left(\frac{\gamma_i{+}1}{2}\right)}{\Gamma\left(\frac{n{+}|\gamma|}{2}+k\right)}L_B^k\delta_\gamma(x).
$$
If $p+|\gamma'|$ is even then weighted functional $P^{\lambda}_{\gamma,+}$ is regular at $\lambda{=}{-}\frac{n+|\gamma|}{2}{-}k$,
$k{\in}\mathbb{N}{\cup}\{0\}$.
\end{theorem}

\begin{proof}
We first consider $\lambda{=}{-}\frac{n{+}|\gamma|}{2}$. Using formula \eqref{Res2} we can write
$$
\underset{\lambda{=}{-}\frac{n{+}|\gamma|}{2}}{\rm res}(P_{\gamma,+}^\lambda,\varphi)_\gamma{=}\Phi\left({-}\frac{n{+}|\gamma|}{2},0\right){=}\frac{\psi_1(0,0)}{4}\int\limits_0^1(1{-}t)^{{-}\frac{n{+}|\gamma|}{2}}t^{\frac{q{+}|\gamma''|}{2}}dt{=}
$$
\begin{equation}\label{Eq1}
=\frac{1}{4}\,\psi_1(0,0)\,\frac{\Gamma\left(\frac{q+|\gamma''|}{2}\right)\Gamma\left(-\frac{n+|\gamma|}{2}+1\right)}{\Gamma\left(-\frac{p+|\gamma'|}{2}+1\right)}.
\end{equation}
From the last formula we can see that if
$p{+}|\gamma'|$ is even then $
\underset{\lambda{=}{-}\frac{n{+}|\gamma|}{2}}{\rm res}(P_{\gamma,+}^\lambda,\varphi){=}0.
$

Now assume that $p+|\gamma'|$ is not integer or  $p+|\gamma'|{\in}\mathbb{N}$ and $p+|\gamma'|{=}2k-1$, $k{\in}\mathbb{N}$   and $q{+}|\gamma''|$ is even.  We have
\begin{equation}\label{Eq2}
\psi_1(0,0)=\psi(0,0)=\varphi(0)\int\limits_{S_p^+}\int\limits_{S_q^+}
\omega^\gamma dS_pdS_q=\varphi(0)|S_1^+(p)|_{\gamma'}|S_1^+(q)|_{\gamma''},
\end{equation}
where
\begin{equation}\label{Plosch}
|S_1^+(p)|_{\gamma'}
=\frac{\prod\limits_{i=1}^p{\Gamma\left(\frac{\gamma'_i{+}1}{2}\right)}}{2^{p-1}\Gamma\left(\frac{p{+}|\gamma'|}{2}\right)},\quad |S_1^+(q)|_{\gamma''}
=\frac{\prod\limits_{i=1}^q{\Gamma\left(\frac{\gamma''_i{+}1}{2}\right)}}{2^{q-1}\Gamma\left(\frac{q{+}|\gamma''|}{2}\right)}
\end{equation}
(see [1], p. 20, formula (1.2.5)).
After some simple calculations, we obtain
$$
\underset{\lambda=-\frac{n+|\gamma|}{2}}{\rm res}(P_{\gamma,+}^\lambda,\varphi)_\gamma
=\frac{(-1)^{\frac{q+|\gamma''|}{2}}}{2^n}\frac{\prod\limits_{i=1}^n\Gamma\left(\frac{\gamma_i{+}1}{2}\right)}{\Gamma\left(\frac{n+|\gamma|}{2}\right)}\,\varphi(0).
$$
Also we have
\begin{equation}\label{Res3}
\underset{\lambda{=}{-}\frac{n{+}|\gamma|}{2}}{\rm res}P_{\gamma,+}^\lambda{=}\frac{(-1)^{\frac{q{+}|\gamma''|}{2}}}{2^n}\frac{\prod\limits_{i{=}1}^n\Gamma\left(\frac{\gamma_i{+}1}{2}\right)}{\Gamma\left(\frac{n{+}|\gamma|}{2}\right)}\delta_\gamma(x).
\end{equation}
Using Green's theorem and the formula  \eqref{Seq2Eq28} we derive
$$
\int\limits_{\{P(x)>0\}^+}\left(\varphi(x)[L_BP^{\lambda+1}(x)]-P^{\lambda+1}(x)[L_B\varphi(x)]\right)x^\gamma dx=0,
$$
therefore
\begin{equation}\label{Seq2Eq31}
(P_{\gamma,+}^\lambda,\varphi)_\gamma=\frac{1}{2(\lambda+1)(2\lambda+n+|\gamma|)}(P_{\gamma,+}^{\lambda+1},L_B\varphi)_\gamma.
\end{equation}
Then $k$-fold iteration of \eqref{Seq2Eq31} leads to
\begin{equation}\label{EqNew}
(P_{\gamma,+}^\lambda,\varphi)_\gamma{=}\frac{(P^{\lambda+k}_{\gamma,+},L_B^k\varphi)_\gamma}{2^{2k}(\lambda+1){...}(\lambda+k)\left(\lambda{+}\frac{n{+}|\gamma|}{2}\right){...}\left(\lambda{+}\frac{n{+}|\gamma|}{2}{+}k{-}1\right)}.
\end{equation}
Consequently
$$
\underset{\lambda=-\frac{n+|\gamma|}{2}-k}{\rm res}(P_{\gamma,+}^\lambda,\varphi)_\gamma=\underset{\lambda=-\frac{n+|\gamma|}{2}-k}{\rm res}(P^{\lambda+k}_{\gamma,+},L_B^k\varphi)_\gamma\times
$$
$$
\times\frac{1}{2^{2k}(\lambda{+}1){...}(\lambda{+}k)\left(\lambda{+}\frac{n{+}|\gamma|}{2}\right){...}\left(\lambda{+}\frac{n{+}|\gamma|}{2}{+}k{-}1\right)}\biggr|_{\lambda{=}{-}\frac{n+|\gamma|}{2}{-}k},
$$
and
$$
\underset{\lambda=-\frac{n+|\gamma|}{2}-k}{\rm res}(P^{\lambda+k}_{\gamma,+},L_B^k\varphi)_\gamma=\underset{\lambda=-\frac{n+|\gamma|}{2}}{\rm res}(P^{\lambda}_{\gamma,+},L_B^k\varphi)_\gamma.
$$
Therefore if $p+|\gamma'|$ is even this residue vanishes. If $p+|\gamma'|$ is not integer or $p+|\gamma'|{\in}\mathbb{N}$ and $p+|\gamma'|{=}2k-1$, $k{\in}\mathbb{N}$ then   \eqref{Res3} gives
$$
\underset{\lambda=-\frac{n+|\gamma|}{2}-k}{\rm res}(P_{\gamma,+}^{\lambda},\varphi)_\gamma=\frac{(-1)^{\frac{q{+}|\gamma''|}{2}}}{2^{n+2k}k!}\frac{\prod\limits_{i{=}1}^n\Gamma\left(\frac{\gamma_i{+}1}{2}\right)}{\Gamma\left(\frac{n{+}|\gamma|}{2}+k\right)}(L_B^k\delta_\gamma(x),\varphi)_\gamma.
$$
This completes the proof of Theorem 2.
\end{proof}

\begin{theorem}
  If $n{+}|\gamma|$ is even and $p+|\gamma'|$ and $q+|\gamma''|$ are also even, $k{\in}\mathbb{N}{\cup}\{0\}$, then function функция $P_{\gamma,+}^\lambda$ has a simple pole in $\lambda{=}{-}\frac{n{+}|\gamma|}{2}{-}k$
with residue
$$
\underset{\lambda{=}{-}\frac{n{+}|\gamma|}{2}{-}k}{\rm res}P_{\gamma,+}^\lambda{=}\frac{1}{\Gamma\left(\frac{n+|\gamma|}{2}+k\right)}\biggl[(-1)^{\frac{n+|\gamma|}{2}+k-1}\delta_{\gamma,1}^{\left(\frac{n+|\gamma|}{2}+k-1\right)}(P)+
$$
$$
+\frac{(-1)^{\frac{q+|\gamma''|}{2}}}{2^{2k}k!}\prod\limits_{i=1}^n\Gamma\left(\frac{\gamma_i+1}{2}\right)\,L_B^k\delta_\gamma(x)\biggr].
$$

If $p{+}|\gamma'|$ and $q{+}|\gamma''|$ are not integer or $p{+}|\gamma'|,q{+}|\gamma''|{\in}\mathbb{N}$ and $p{+}|\gamma'|=2m-1,$ $q{+}|\gamma''|{=}2k-1$, $m,k{\in}\mathbb{N}$ then function $P_{\gamma,+}^\lambda$ a pole of order two at $\lambda{=}{-}\frac{n{+}|\gamma|}{2}{-}k$. Coefficients
$c_{-2}^{(k)}$ and $c_{-1}^{(k)}$ of expansion of function $P_{\gamma,+}^\lambda$ in Laurent series at $\lambda=-\frac{n+|\gamma|}{2}-k$ are expressed by the formulas
$$
c_{-1}^{(0)}=\frac{1}{\Gamma\left(\frac{n+|\gamma|}{2}+k\right)}\biggl[(-1)^{\frac{n+|\gamma|}{2}+k-1}\delta_{\gamma,1}^{\left(\frac{n+|\gamma|}{2}+k-1\right)}(P)+
\frac{(-1)^{\frac{n+|\gamma|}{2}-1}}{2^{2k}k!}\times
$$
$$\times\prod\limits_{i=1}^n\Gamma\left(\frac{\gamma_i{+}1}{2}\right)\sin\left(\frac{p{+}|\gamma'|}{2}\pi\right)\left(\psi\left(\frac{p{+}|\gamma'|}{2}\right){-}\psi\left(\frac{n{+}|\gamma|}{2}\right)\right)L_B^k\delta_\gamma(x)\biggr],
$$
$$
c_{-2}^{(k)}=(-1)^{\frac{n+|\gamma|}{2}+1}\frac{\sin\frac{\pi (p+|\gamma'|)}{2}\prod\limits_{i=1}^n\Gamma\left(\frac{\gamma_i+1}{2}\right)}{2^{n+2k}k!
\pi\Gamma\left(\frac{n+|\gamma|+k}{2}\right)}L_B^k\delta_\gamma(x),
$$
where $\psi(x)=\frac{\Gamma'(x)}{\Gamma(x)}$.
\end{theorem}

\begin{proof}
Let $n{+}|\gamma|$ is even and $\lambda{=}{-}\frac{n{+}|\gamma|}{2}{-}k,$ $k{\in}\mathbb{N}{\cup}\{0\}.$
We express this $(P_{\gamma,+}^\lambda,\varphi)_\gamma$ in the form
\begin{equation}\label{Razl1}
(P_{\gamma,+}^\lambda,\varphi)_\gamma{=}\frac{1}{\lambda{+}k}\int\limits_{0}^\infty
u^{\lambda{+}\frac{n{+}|\gamma|}{2}{-}1}\Phi_0(u)du{+}\int\limits_{0}^\infty
u^{\lambda{+}\frac{n{+}|\gamma|}{2}{-}1}\Phi_1(\lambda,u)du,
\end{equation}
where $\Phi_0(u){=}\underset{\lambda{=}{-}\frac{n{+}|\gamma|}{2}{-}k}{\rm res}\Phi(\lambda,u)$ and	 $\Phi_1(\lambda,u)$ is a regular at $\lambda{=}{-}\frac{n{+}|\gamma|}{2}{-}k$ function. By virtue of the proposal each integral in \eqref{Razl1} may have at $\lambda{=}{-}\frac{n{+}|\gamma|}{2}{-}k$ a simple pole therefore function $(P_{\gamma,+}^\lambda,\varphi)_\gamma$ may have a pole of order two at
 $\lambda{=}{-}\frac{n{+}|\gamma|}{2}{-}k$.
In the neighborhood of such a point we may expand $P_{\gamma,+}^\lambda$ in the Laurent
series
$$
P_{\gamma,+}^\lambda=\frac{c_{-2}^{(k)}}{\left(\lambda+\frac{n+|\gamma|}{2}+k\right)^2}+\frac{c_{-1}^{(k)}}{\lambda+\frac{n+|\gamma|}{2}+k}+...\, .
$$
Let us find $c_{-1}^{(k)}$, $c_{-2}^{(k)}$.
We have
$$
(c_{-2}^{(k)},\varphi)_\gamma=\underset{\lambda=-\frac{n+|\gamma|}{2}-k}{\rm res}\int\limits_{0}^\infty
u^{\lambda+\frac{n+|\gamma|}{2}-1}\Phi_0(u)du=\frac{1}{k!}\Phi_0^{(k)}(0).
$$
If $k=0$ then $c_{-2}^{(0)}=\Phi_0(0).$
According to \eqref{Phi}
$$
\Phi_0(0)=\frac{1}{4}\psi_1(0,0)\underset{\lambda=-\frac{n+|\gamma|}{2}}{\rm res}\int\limits_{0}^1 (1-t)^\lambda t^{\frac{q+|\gamma''|-2}{2}}dt=
$$
$$
=\psi_1(0,0)\underset{\lambda=-\frac{n+|\gamma|}{2}}{\rm res}\frac{\Gamma\left(\frac{q+|\gamma''|}{2}\right)\Gamma(\lambda+1)}{4\Gamma\left(\lambda+\frac{q+|\gamma''|}{2}+1\right)}.
$$
Considering that $\psi_1(0,0){=}\varphi(0)|S_1^+(p)|_{\gamma'}|S_1^+(q)|_{\gamma''}$ where $|S_1^+(p)|_{\gamma'}$ and $|S_1^+(q)|_{\gamma''}$ were determined in \eqref{Plosch} we obtain
$$
(c_{-2}^{(0)},\varphi)_\gamma=$$
$$=\frac{(-1)^{\frac{n+|\gamma|}{2}+1}B\left(\frac{p+|\gamma'|}{2},\frac{q+|\gamma''|}{2}\right)}{4\pi}\sin\frac{\pi (p+|\gamma'|)}{2}|S_1^+(p)|_{\gamma'}|S_1^+(q)|_{\gamma''}\varphi(0).
$$
When $p+|\gamma'|$ is even (in this case $q+|\gamma''|$ is also even) we have $c_{-2}^{(k)}=0$ i.e. function $(P_{\gamma,+}^\lambda,\varphi)_\gamma$ has a simple pole at  $\lambda=-\frac{n+|\gamma|}{2}$. If  $p+|\gamma'|$ is not integer or $p+|\gamma'|{\in}\mathbb{N}$ and $p+|\gamma'|{=}2k-1$, $k{\in}\mathbb{N}$ then
$$
c_{-2}^{(0)}=(-1)^{\frac{n+|\gamma|}{2}+1}\frac{\sin\frac{\pi (p+|\gamma'|)}{2}\prod\limits_{i=1}^n\Gamma\left(\frac{\gamma_i+1}{2}\right)}{2^n
\pi\Gamma\left(\frac{n+|\gamma|}{2}\right)}\delta_\gamma(x).
$$
As well as in Theorem 2 we obtain that if
$p+|\gamma'|$ and $q+|\gamma''|$ are even then function $P_{\gamma,+}^\lambda$ has a simple pole at $\lambda=-\frac{n+|\gamma|}{2}-k$. If $p+|\gamma'|$ and $q+|\gamma''|$ are not integer or $p{+}|\gamma'|,q{+}|\gamma''|{\in}\mathbb{N}$ and $p{+}|\gamma'|{=}2m{-}1,$ $q{+}|\gamma''|{=}2k{-}1$, $m,k{\in}\mathbb{N}$ then
$$
c_{-2}^{(k)}=(-1)^{\frac{n+|\gamma|}{2}+1}\frac{\sin\frac{\pi (p+|\gamma'|)}{2}\prod\limits_{i=1}^n\Gamma\left(\frac{\gamma_i+1}{2}\right)}{2^{n+2k}k!
\pi\Gamma\left(\frac{n+|\gamma|+k}{2}\right)}L_B^k\delta_\gamma(x).
$$

Let's find $c_{-1}^{(k)}$. We have
$$
(c_{-1}^{(k)},\varphi)=\int\limits_0^\infty u^{-k-1}\Phi_0(u)du+$$
$$+\underset{\lambda=-\frac{n+|\gamma|}{2}-k}{\rm res}\int\limits_0^\infty  u^{\lambda+\frac{n+|\gamma|}{2}-1}\Phi_1\left(-\frac{n+|\gamma|}{2}-k,u\right)du.
$$
Since $
\Phi_0(u)=\underset{\lambda=-k}{\rm res}\Phi(\lambda,u)$
then using the formulas \eqref{Seq2Eq17} and  \eqref{Seq2Eq24} we obtain
$$
\int\limits_0^\infty u^{-k-1}\Phi_0(u)du=\frac{(-1)^{\frac{n+|\gamma|}{2}+k-1}}{\Gamma\left(\frac{n+|\gamma|}{2}+k-1\right)}\left(\delta_{\gamma,1}^{(\frac{n+|\gamma|}{2}+k-1)}(P),\varphi\right)_\gamma.
$$
Thus
$$
\underset{\lambda=-\frac{n+|\gamma|}{2}-k}{\rm res}\int\limits_0^\infty  u^{\lambda+\frac{n+|\gamma|}{2}-1}\Phi_1\left(-\frac{n+|\gamma|}{2}-k,u\right)du=
$$
$$=\frac{1}{k!}\frac{\partial^k\Phi_1\left(-\frac{n+|\gamma|}{2}-k,u\right)}{\partial u^k}\biggr|_{u=0}=(\alpha^{(k)}_\gamma,\varphi)_\gamma
$$
and
$$
c_{-1}^{(k)}=\frac{(-1)^{\frac{n+|\gamma|}{2}+k-1}}{\Gamma\left(\frac{n+|\gamma|}{2}+k-1\right)}\delta_{\gamma,1}^{(\frac{n+|\gamma|}{2}+k-1)}(P)+\alpha^{(k)}_\gamma.
$$
For $k=0$ we obtain
$$
(\alpha^{(0)}_\gamma,\varphi)_\gamma=\Phi_1\left(-\frac{n+|\gamma|}{2},0\right).
$$
In order to find $\Phi_1\left(-\frac{n+|\gamma|}{2},0\right)$ we consider  $\Phi(\lambda,0)$. Using \eqref{Eq1}, \eqref{Eq2} and \eqref{Plosch} we obtain
$$
\Phi(\lambda,0)=\varphi(0)\frac{\Gamma(\lambda+1)\prod\limits_{i=1}^n\Gamma\left(\frac{\gamma_i+1}{2}\right)}{2^n\Gamma\left(\frac{p+|\gamma'|}{2}\right)\Gamma\left(\lambda+\frac{q+|\gamma''|}{2}+1\right)}.
$$
Taking into account the formula $\Gamma(1-x)\Gamma(x)=\frac{\pi}{\sin\pi x}$ we can write
$$
\Phi(\lambda,0)=\frac{\sin\pi\left(\lambda+\frac{q+|\gamma''|}{2}\right)}{\sin\pi\lambda}\frac{\Gamma\left(-\lambda-\frac{q+|\gamma''|}{2}\right)\prod\limits_{i=1}^n\Gamma\left(\frac{\gamma_i+1}{2}\right)}{\Gamma\left(\frac{p+|\gamma'|}{2}\right)\Gamma(-\lambda)}\varphi(0).
$$
If $p{+}|\gamma'|$ and $q{+}|\gamma''|$ are even then $$\lim\limits_{\lambda\rightarrow-\frac{n+|\gamma|}{2}}\frac{\sin\pi\left(\lambda{+}\frac{q{+}|\gamma''|}{2}\right)}{\sin\pi\lambda}{=}(-1)^{\frac{q{+}|\gamma''|}{2}},$$ hence function $\Phi(\lambda,0)$ is regular at $\lambda{=}{-}\frac{n{+}|\gamma|}{2}$ and
$$
\Phi_1\left(-\frac{n+|\gamma|}{2},0\right){=}\Phi\left(-\frac{n+|\gamma|}{2}\right)
$$ whence
$$
(\alpha^{(0)}_\gamma,\varphi)_\gamma=(-1)^{\frac{q+|\gamma''|}{2}}\frac{\prod\limits_{i=1}^n\Gamma\left(\frac{\gamma_i+1}{2}\right)}{\Gamma\left(\frac{n+|\gamma|}{2}\right)}\varphi(0).
$$
If $p{+}|\gamma'|$ and $q{+}|\gamma''|$ are not integer or  $p{+}|\gamma'|,q{+}|\gamma''|{\in}\mathbb{N}$ and $p{+}|\gamma'|{=}2m{-}1,$ $q{+}|\gamma''|{=}2k{-}1$, $m,k{\in}\mathbb{N}$ then $\Phi(\lambda,0)$ has a pole at $\lambda=-\frac{n+|\gamma|}{2}$. In this case
$$
(\alpha^{(0)}_\gamma,\varphi)_\gamma=\Phi_1\left(-\frac{n+|\gamma|}{2},0\right)=(-1)^{\frac{n+|\gamma|}{2}-1}\prod\limits_{i=1}^n\Gamma\left(\frac{\gamma_i+1}{2}\right)\times
$$
$$
\times\frac{\sin\left(\frac{p+|\gamma'|}{2}\pi\right)\left(\psi\left(\frac{p+|\gamma'|}{2}\right)-\psi\left(\frac{n+|\gamma|}{2}\right)\right)}{\Gamma\left(\frac{n+|\gamma|}{2}\right)}\,\varphi(0),
$$
where $
\psi(x)=\frac{\Gamma'(x)}{\Gamma(x)}.$
We obtain
$$
c_{-1}^{(0)}=\frac{1}{\Gamma\left(\frac{n+|\gamma|}{2}\right)}\left[(-1)^{\frac{n+|\gamma|}{2}-1}\delta_{\gamma,1}^{\left(\frac{n+|\gamma|}{2}-1\right)}(P)+\theta\delta_\gamma(x)\right],
$$
with a value
$$\theta{=}(-1)^{\frac{q+|\gamma''|}{2}}\prod\limits_{i=1}^n\Gamma\left(\frac{\gamma_i+1}{2}\right)$$ if $p{+}|\gamma'|$ and $q{+}|\gamma''|$ are even.
If $p{+}|\gamma'|$ and $q{+}|\gamma''|$ are not integer or $p{+}|\gamma'|,q{+}|\gamma''|{\in}\mathbb{N}$ and $p{+}|\gamma'|{=}2m{-}1,$ $q{+}|\gamma''|{=}2k{-}1$, $m,k{\in}\mathbb{N}$ then
$$\theta=(-1)^{\frac{n+|\gamma|}{2}-1}\prod\limits_{i=1}^n\Gamma\left(\frac{\gamma_i{+}1}{2}\right)\sin\left(\frac{p{+}|\gamma'|}{2}\pi\right)\times
$$
$$\times\left(\psi\left(\frac{p+|\gamma'|}{2}\right){-}\psi\left(\frac{n+|\gamma|}{2}\right)\right).$$
Finally, in order to obtain $c_{-1}^{(k)}$ for arbitrary $k$, we again use the
 formula \eqref{EqNew}. This proves the desired result.
\end{proof}



\vskip 1cm
Elina Leonidovna Shishkina\\
Voronezh State University\\
1 Universitetskaya pl., Voronezh 394006, Russia\\
E-mail: ilina\_dico@mail.ru

\end{document}